\documentclass[11pt]{article}
\usepackage{amssymb,mathtools,amsthm} %%% amsmath
\usepackage{graphicx}
\newtheorem{theorem}{Theorem}
\newtheorem{conj}[theorem]{Conjecture}
\newtheorem{lemma}[theorem]{Lemma}
\newtheorem{claim}{Claim}

\let\oldproofname=\proofname
\renewcommand{\proofname}{\rm\bf{\oldproofname}}
\def\coef1{{(3/2)}}
\def\supp{{\rm support}}
\def\Tr{{\rm Tr}}
\def\w{f}

\begin{document}

\title{The minimum number of triangular edges and a symmetrization method for multiple graphs}
\author{Zolt\'an F\"uredi
\thanks{ Alfr\'ed R\'enyi Institute of Mathematics, 13--15 Re\'altanoda Street, 1053 Budapest, Hungary. \newline
E-mail: {\tt z-furedi@illinois.edu}.\newline\indent
Research was supported in part by grant (no. K104343) from the
National Research, Development and Innovation Office – NKFIH,
%%% Research supported in part by the Hungarian National Science Foundation OTKA 104343,
 by the Simons Foundation Collaboration Grant \#317487,
and by the European Research Council Advanced Investigators Grant 267195.}
\and Zeinab Maleki
\thanks{ Department of Mathematical Sciences, Isfahan University of Technology, Isfahan 84156-83111, Iran. \newline E-mail: {\tt zmaleki@math.iut.ac.ir}
%  \newline
% This work was done while the authors visited the Department of Mathematics and Computer Science, Emory University, Atlanta, GA, USA.
\newline\indent
{\it 2010 Mathematics Subject Classifications:}
05C35, 05C22, 05D99. \hfill \jobname
\newline\indent
{\it Key Words}:  Tur\'an number, triangles, extremal graphs, symmetrization.  \hfill\today
}
}
\date{This work was done while the authors visited the Department of Mathematics and Computer Science, Emory University, Atlanta, GA, USA.}

\maketitle

%%%%%%%%%%%%%%%%%%%%%%%%%%%%%%%%%%%%%%%%%%%%%%%%%%
%%%%%%%%%%%%%%%%%%%%%%%%%%%%%%%%%%%%%%%%%%%%%%%%%%
\begin{abstract}
We give an asymptotic formula for the minimum number of edges contained in triangles in a graph having $n$ vertices and $e$ edges.
Our main tool is a generalization of Zykov's symmetrization method that can be applied for several graphs simultaneously.
\end{abstract}
%%%%%%%%%%%%%%%%%%%%%%%%%%%%%%%%%%%%%%%%%%%%%%%%%%
%%%%%%%%%%%%%%%%%%%%%%%%%%%%%%%%%%%%%%%%%%%%%%%%%%

\section{Graphs with few triangular edges}

Erd\H os, Faudree, and Rousseau~\cite{EFR92} showed that a graph on $n$ vertices and at least $\lfloor n^2/4\rfloor+1$ edges has at least
 $2\lfloor n/2\rfloor +1$ edges in triangles.
To see that this result is sharp, consider the graph obtained by adding one edge to the larger side of the complete bipartite graph $K_{\lceil n/2 \rceil, \lfloor n/2\rfloor}$.
We consider a more general problem, where the number of edges may be larger than $\lfloor n^2/4\rfloor+1$. Given a graph $G$, denote by $\Tr(G)$ the number of edges of $G$ contained in triangles, and let $\Tr(n,e):=\min \{ \Tr(G): |V(G)|=n, \, e(G)=e\}$.
With this notation the above result of Erd\H os, Faudree, and Rousseau can be reformulated as
\begin{equation}\label{eq1}
    \Tr(n,\lfloor n^2/4\rfloor+1)= 2\lfloor n/2\rfloor +1.
    \end{equation}
Note that $\Tr(n,e)=0$ whenever $e\leq n^2/4$, because in that case there exist triangle-free (even bipartite) graphs with $n$ vertices and $e$ edges.
% $K_{\lceil n/2 \rceil, \lfloor n/2\rfloor}$.
To avoid trivialities, we usually implicitly assume that $e> n^2/4$.
% and $n\geq 5$ (although it is not necessary in most cases).
%Obviously, $\Tr(4,4)=0$ and $\Tr(4,e)=e$ for $e\in \{5,6 \}$.

%Let define some graphs which help us to construct a class of graphs with many non-triangular edges.
Given integers $a$, $b$ and $c$, ($a \geq 2$), we define a family of graphs ${\mathcal G}(a,b,c)$ as follows, see Figure~\ref{G(a,b,c)} below.
The vertex set $V$ of a graph $G$ in this class has a partition $V = A \cup B \cup C$ where $|A|=a$, $|B|=b$, and $|C|=c$, such that $B$ and $C$ are independent sets, $B\cup C$ induces a complete bipartite graph $K_{b,c}$, the vertices of $C$ have neighbors only in $B$, and
$G[A]$ and $G[A,B]$ are `almost complete graphs', namely, they span more than ${|A|-1 \choose 2}+ |A||B|$ edges.
The edges of $G[B,C]$ are the non-triangular edges.

\begin{figure}[ht]\label{G(a,b,c)}
\begin{center}
\includegraphics[scale=.5]{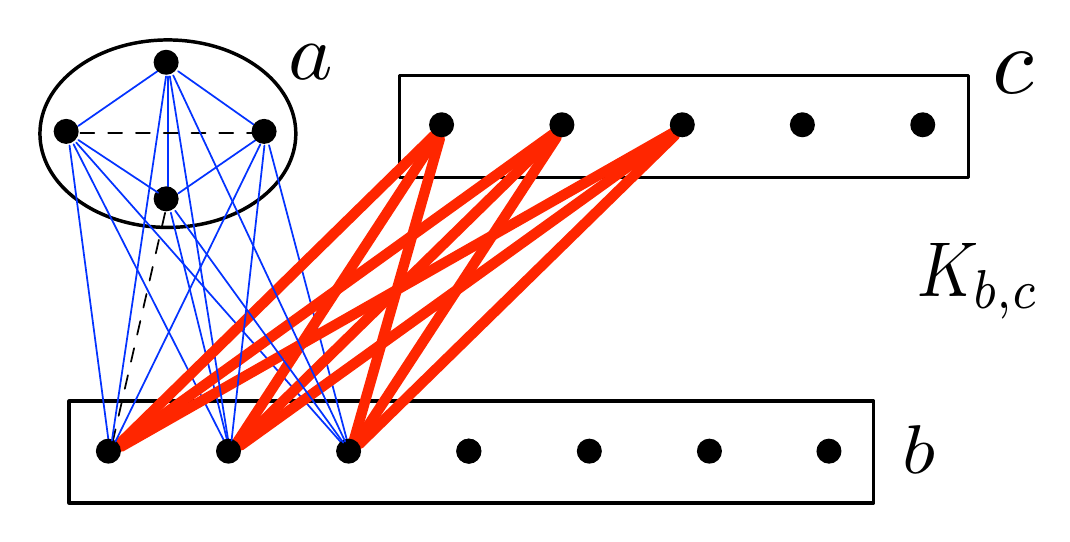}
\end{center}
%\vskip -4mm
\caption{\quad A graph from ${\mathcal G}(a,b,c)$.}
\end{figure}

Given integers $n\geq 3$ and $n^2/4 < e \leq {n \choose 2}$, we define a class of graphs, ${\mathcal G}(n,e)$,  with many non-triangular edges as follows.
Put a graph $G \in {\mathcal G}(a,b,c)$ into the class ${\mathcal G}(n,e)$ if it has $n$ vertices and $e$ edges.
% it has the minimum number of triangular edges among these type of graphs, i.e., $\Tr(G) \leq \Tr(G')$ holds for  $G'\in {\mathcal G}(a',b',c')$ with $|a'|+|b'|+|c'|=n$ and $e(G')=e$.
Define $g(n,e)$ as $\min\{\Tr (G) : G \in {\mathcal G}(n,e)\}$.
We have
\begin{equation}\label{eq2}
  \Tr(n,e)\leq g(n,e)=\min\{e-bc: a+b+c=n,\, a,b,c \in \mathbb{N}\cup \{0\},\, {a\choose 2}+ab+bc \ge e\}.
  \end{equation}

We believe that one can extend the Erd\H os, Faudree, Rousseau theorem~\cite{EFR92} as follows.
\begin{conj}\label{conj1}
Suppose that $G$ is an $n$-vertex graph with $e$ edges, such that $e> n^2/4$ and it has the minimum number of triangular edges,
 i.e., $\Tr(G)= \Tr(n,e)$.
Then $G\in {\mathcal G}(n,e)$.
\end{conj}
In particular, we conjecture that $\Tr(n,e)=g(n,e)$.
We prove a slightly weaker result.

\begin{theorem}\label{thm:th2}
For $e>n^2/4$ we have
$g(n, e)-\coef1 n \leq \Tr(n,e) \leq g(n,e)$.
\end{theorem}

Our main tool, presented in Section~\ref{sec:sym_mtd}, is a new symmetrization method, a generalization of previous results by Zykov and Motzkin and Straus such that it can be applied to more than one graph simultaneously.

In Section~\ref{sec:triangle}, we use the new symmetrization method to prove a lemma about triangular edges of a given graph. In Section~\ref{sec:lb}, using the lemma of Section~\ref{sec:triangle} we complete the proof of Theorem~\ref{thm:th2}. In Section~\ref{sec:future} we introduce more problems for future research, our method can be used to solve some of them (see~\cite{FM14+}).
%%%%%%%%%%%%%%%%%%%%%%%%%%%%%%%%%%%%%%%%%%%%%%%%%%
%%%%%%%%%%%%%%%%%%%%%%%%%%%%%%%%%%%%%%%%%%%%%%%%%%

\section{The symmetrization method}
\label{sec:sym_mtd}

In this section, we describe Zykov's symmetrisation process~\cite{Zykov49}. It starts with a $K_{p}$-free graph $G$ with vertex set $\{v_1,\dots,v_n\}$ and at each step takes two nonadjacent vertices $v_i$ and $v_j$ such that $\deg(v_i) > \deg (v_j)$ and replaces all edges incident to $v_j$ by new edges incident to $v_j$ and to the neighborhood $N(v_i)$. We do the same if $\deg(v_i) = \deg (v_j)$,   $N(v_i) \neq N(v_j)$ and $i <j$.
Symmetrization does not increase the size of the largest clique and does not decrease the number of edges.
When the process terminates it yields a complete multipartite graph with at most $p-1$ parts.

This way Zykov~\cite{Zykov49} gave a proof of Tur\'an's theorem which states that the number of edges of a $K_p$-free graph is at most as large as in a complete $(p-1)$-partite graph with almost equal parts.
It seems that this method cannot be used directly to determine $\Tr(n,e)$ because we need to increase simultaneously the number of edges and the number of non-triangular edges. In the rest of the section this
method will be generalized to settings involving more than one graph.

Let us recall a continuous version of Zykov's symmetrisation method, due to Motzkin and Straus~{\rm \cite{MS65}}.
Given a graph $G$ with vertex set $\{ v_1, \dots, v_n\}$ define a real polynomial
$$\w(G,{\mathbf x}):=\sum \{ x_ix_j: v_iv_j\in E \}.$$
Define a simplex $S_n:=\{{\mathbf x}\in {\mathbb{R}}^n: \forall x_i\geq 0$ and $\sum x_i=1\}$.
Let $\w(G):=\max \{ \w(G,{\mathbf x}): {\mathbf x}\in S_n\}$.
Motzkin and Straus~\cite{MS65} provided an alternative proof of an asymptotic version of Tur\'{a}n's theorem by observing a remarkable connection between the clique number, $\omega(G)$,  and $\w(G)$.
They proved that   $\w(G)=(\omega-1)/(2\omega)$.
Their main tool was a continuous version of Zykov's symmetrization as follows.

\begin{theorem}{\rm{(Motzkin and Straus~\cite{MS65}})\label{thm:MS}}\enskip
Given a graph $G$ on $n$ vertices and a vector ${\mathbf x}\in S_n$, there exists ${\mathbf y}\in S_n$ such that $\w(G,{\mathbf x})\leq \w(G,{\mathbf y})$ and $\supp ({\mathbf y})$ induces a complete subgraph.
\end{theorem}

We generalize this result so that it can be applied simultaneously for several graphs.

\begin{theorem}\label{thm:ind}
Let $G$ be a graph on $n$ vertices and let $G_1, G_2, \dots, G_{d}$ be subgraphs of $G$ with the same vertex set. For every ${\mathbf x}\in S_{n}$ there exists a subset $K\subseteq V(G)$ and a vector ${\mathbf y}\in S_n$ with support $K$ such that $\w(G_i,{\mathbf x})\leq \w(G_i,{\mathbf y})$ for every $1 \le i \le d$ and $\alpha(G[K])\leq d$.
\end{theorem}

To prove Theorem~\ref{thm:ind} we need the following lemma.
\begin{lemma}\label{lem:geometry}
Suppose that ${\mathbf a}_1, \dots, {\mathbf a}_{d}\in {\mathbb{R}}^{d+1}$.
Then there exists a non-zero vector ${\mathbf z}\in {\mathbb{R}}^{d+1}$
  such that ${\mathbf a}_i^T {\mathbf z} \geq 0$ for every $1 \leq i \leq d$ and the sum of the coordinates is $0$, namely $\sum_{1\leq i\leq d+1} z_i=0$.
\end{lemma}
\begin{proof}{
Let ${\mathbf j} \in {\mathbb{R}}^{d+1}$ be the all $1$ vector and define the matrix  $A$ as $\{{\mathbf a}_1,\dots,{\mathbf a}_d,{\mathbf j}\}$.
If $\det(A) = 0$, then there are non-trivial solutions of $A^T{\mathbf z}=\mathbf 0$.
If $\det(A) \neq 0$ define ${\mathbf a}:= (1,\dots,1,0)^T \in {\mathbb{R}}^{d+1}$.
There is a unique solution ${\mathbf z}$ of $A^T{\mathbf z}={\mathbf a}$. Clearly, ${\mathbf z} \neq 0$ so we are done.
}\end{proof}

\begin{proof}[\bf Proof of Theorem~\ref{thm:ind}.]{
Let ${\mathbf y} \in S_{n}$ be a vector whose support has minimum size among vectors ${\mathbf y}' \in S_{n}$ satisfying $\w(G_i,{\mathbf x})\leq \w(G_i,{\mathbf y}')$ for every $1 \le i \le d$.
%$$|\supp({\mathbf y})|=\min\{|\supp({\mathbf y}')|: \w(G_i,{\mathbf x})\leq \w(G_i,{\mathbf y}'), \forall  1 \le i \le d\}.
%$$
If $\{v_1,v_2,\dots,v_{d+1}\} \subseteq \supp({\mathbf y})$ is an independent set, then for any
\linebreak%%%%%%%%%%%%%%%%%%%%%%%%%%%
${\mathbf z}=(z_1, \dots, z_{d+1},0,0,\dots)^T \in {\mathbb{R}}^n$, $t\in {\mathbb{R}}$, and $1 \le i \le d$ we have
 $\w(G_i,{\mathbf y}+t{\mathbf z})= \w(G_i,{\mathbf y})+t({\mathbf a}_i^T{\mathbf z})$ for some ${\mathbf a}_i \in {\mathbb{R}}^{d+1}$.
Here ${\mathbf a}_i$ depends only on $G_i$ and ${\mathbf y}$, not on ${\mathbf z}$ or $t$.
Apply Lemma~\ref{lem:geometry} to obtain a non-zero vector ${\mathbf z}=(z_1, \dots, z_{d+1},0,0,\dots)^T$ with $\sum_{1\leq i\leq d+1} z_i=0$ and ${\mathbf a}_i^T{\mathbf z}\geq 0$ for $1 \le i \le d$. Choosing an appropriate $t>0$ we have ${\mathbf y}+t{\mathbf z}\in S_n$ and $\supp({\mathbf y}+t{\mathbf z}) \subseteq \supp({\mathbf y})-\{v_j\}$ for some $1 \le j \le d+1$. This is a contradiction, so ${\mathbf y}$ has the desired property.
}\end{proof}

%%%%%%%%%%%%%%%%%%%%%%%%%%%%%%%%%%%%%%%%%%%%%%%%%%%%%%%%%%%%%%%%%%%%%

\section{Maximizing the weight of non-triangular edges in a weighted graph}
\label{sec:triangle}

%In this section, by using Theorem~\ref{thm:ind}, we prove a strong lemma concerning triangular edges, as follows.

\begin{lemma}\label{lem:triw}
Let $G_1$ be a graph on $n$ vertices $\{v_1,\dots,v_n\}$ and let $G_2$ be a subgraph of $G_1$ whose edges are some of the non triangular edges of $G_1$,  $E(G_2) \neq \emptyset$. For every ${\mathbf x}\in S_{n}$
there exists a subset $K\subseteq V$ and a vector ${\mathbf y}\in S_n$ with support $K$ such that $\w(G_1,{\mathbf x})\leq \w(G_1,{\mathbf y})$ and $\w(G_2,{\mathbf x})\leq \w(G_2,{\mathbf y})$. Furthermore, the graph $H:=G_1[K]$ contains exactly one edge $e$ of $G_2$ and $H \setminus V(e)$ is a complete graph.
%
%Let $G_1$ and $G_2$ be graphs on a common vertex set $V:=\{ v_1, \dots, v_n \}$.
%Suppose that $\emptyset\neq E(G_2)\subset E(G_1)$ and no edge in $G_2$ can appear on a triangle of\, $G_1$.
%Take any ${\mathbf x}\in S_{n}$.
%Then there exists a subset $K\subseteq V$  and a vector ${\mathbf y}\in S_n$ with support $K$
% such that  $\w(G_i,{\mathbf x})\leq \w(G_i,{\mathbf y})$ for $i\in \{ 1,2\}$, $\alpha(H)\leq 2$ where the graph $H$ is the subgraph of $G_1$ induced by $K$.
%\\
%Moreover, $H$ contains only a single $G_2$ edge, $v_1v_2$, and $H[(K\setminus \{v_,v_2\})]$ is a complete graph.
\end{lemma}

\begin{proof}{
By Theorem~\ref{thm:ind}, we know that there is a ${\mathbf y}\in S_{n}$ such that $\w(G_1,{\mathbf x})\leq \w(G_1,{\mathbf y})$, $\w(G_2,{\mathbf x})\leq \w(G_2,{\mathbf y})$ and $\alpha(H)\leq 2$. Let ${\mathbf y}=(y_1,\dots,y_n)$ be such a vector whose support has minimal size.
%with $|\supp({\mathbf y})|$ is minimized.
We claim that $K:= \supp({\mathbf y})$ satisfies the required properties.
First we show that the structure of $G_2[K]$ is rather simple, then we show that by finding an appropriate ${\mathbf y}'$
 one can further reduce $K$ if $G_2[K]$ has two or more edges.

Recall that $\frac{\partial}{\partial z_k}f$ stands for the partial derivative of the function $f(z_1, z_2, ...,z_n)$ with respect to the variable $z_k$.
Suppose that $v_k$ and $v_h\in K$ are nonadjacent vertices such that
\begin{equation}\label{eq3}
\frac{\partial}{\partial y_k}\w(G_1, {\mathbf y}) \geq \frac{\partial}{\partial y_h}\w(G_1, {\mathbf y})
  {\rm  \, \, \, \, and \, \,\, \, }   \frac{\partial}{\partial y_k}\w(G_2, {\mathbf y}) \geq \frac{\partial}{\partial y_h}\w(G_2, {\mathbf y}).
  \end{equation}
In other words, $\sum\{ y_\ell: v_kv_\ell\in E(G_i[K])\}\geq \sum\{ y_\ell: v_hv_\ell\in E(G_i[K])\}$ for $i=1,2$.
Define the vector ${\mathbf y}' \in S_n$ by %${\mathbf y}':={\mathbf y} +y_h({\mathbf e}_k-{\mathbf e}_h)\in S_n$, i.e.,
%$$
%y'_l =
%\begin{cases}
%y_k+y_h & l=k \\
%0 & l=h \\
%y_l & \mbox{otherwise}.
%\end{cases}
%$$
\[
y'_{\ell} =
\left\{
\begin{array}{ll}
y_k+y_h & {\ell}=k \\
0 & {\ell}=h \\
y_{\ell} & \mbox{otherwise}.
\end{array}
\right.
\]
%its $k$'th coordinate $y_k'=y_k+y_h$, its $h$'th coordinate is $0$, and $y_\ell'=y_\ell$ otherwise.
We have  $\w(G_i,{\mathbf y})\leq \w(G_i,{\mathbf y}')$ for $i\in \{ 1,2\}$ and $\supp({\mathbf y}')=K\setminus \{ v_h\}$, a contradiction.
We conclude that condition (\ref{eq3}) does not hold.

Without loss of generality, we may suppose that $v_1v_2$ is a $G_2$-edge  of $H$. From now on, in this section if we talk about 'edges', 'degrees' etc., then we always mean $H$-edges, degree in $H$, etc., except if it is otherwise stated.

If $\w(G_1,{\mathbf y})\leq  1/4$ then define   ${\mathbf y}'=(1/2, 1/2, 0,\dots, 0)$.
We obtain $\w(G_2,{\mathbf y})\leq \w(G_1,{\mathbf y})\leq 1/4 =\w(G_1,{\mathbf y}')=\w(G_2,{\mathbf y}')$. This implies $K=\{ 1,2 \}$ and we are done.
So from now on, we suppose that $\w(G_1,{\mathbf y}) >  1/4$.
Then the Motzkin-Straus theorem implies that the graph $H$ is not triangle-free.

\begin{claim}\label{cl:1}
There are no two adjacent edges of $G_2[K]$.
\end{claim}

\begin{proof}[Proof of Claim~\ref{cl:1}.]
%
%To prove that $H$ contains a single edge of $G_2$ we assume, on the contrary,  that $H$ contains more than one edge from $G_2$ and then reach a contradiction.
Assume, to the contrary, that $v_1v_2$ and $v_1v_3 \in E(H)$ are $G_2$ edges.
We claim that
\begin{multline}
${}$\quad\quad \mbox{$v_2$ and $v_3$ are non-adjacent, $\deg(v_1)=2$, and}\\
 \mbox{$H\setminus \{ v_1,v_2,v_3\}$ is a complete graph.}\quad\quad${}$\label{eq3.2}\end{multline}
Indeed, $v_2$ and $v_3$ are non-adjacent, otherwise the triangle $v_1v_2v_3$ contains $G_2$ edges.
Suppose, to the contrary, that $|N(v_1)| > 2$, i.e., there exists a vertex $v_4 \neq v_2, v_3$, such that $v_1v_4\in E(H)$. Since $\alpha(H) \le 2$ and $v_2v_3 \notin E(H)$, without loss of generality, $v_3v_4 \in E(H)$.
Then the triangle $v_1v_3v_4$ contains a $G_2$ edge (namely $v_1v_3$), a contradiction, so we must have $N(v_1)=\{v_2,v_3\}$.
Finally, the condition $\alpha(H)\leq 2$ implies that $K\setminus (N(v_1) \cup \{ v_1\})$ induces a complete graph (cf., Figure~\ref{fig:adj1}).

The statement~(\ref{eq3.2}) already implies that the structure of $G_2$ edges is rather simple in $H$.
Using Condition~(\ref{eq3}) and other techniques, we reach a contradiction, considering three possible cases.

\paragraph{\bf Case~$1a$.}\enskip
Assume that there is no $G_2$ edge connecting $\{ v_1,v_2,v_3\}$ to
  $K\setminus \{ v_1,v_2,v_3\}$.

Then $ \frac{\partial}{\partial y_2}\w(G_2, {\mathbf y}) = \frac{\partial}{\partial y_3}\w(G_2, {\mathbf y})$ (namely, both are $y_1$).
Since $v_2$ and $v_3$ are non-adjacent the conditions of (\ref{eq3}) hold,  a contradiction.

\begin{figure}[ht]
    \centering
    {\includegraphics[width=0.32\textwidth]{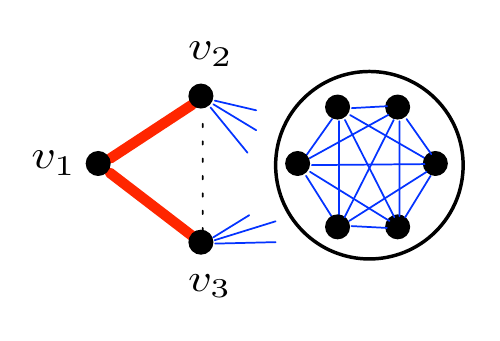}}
    \hspace{16mm}
    {\includegraphics[width=0.32\textwidth]{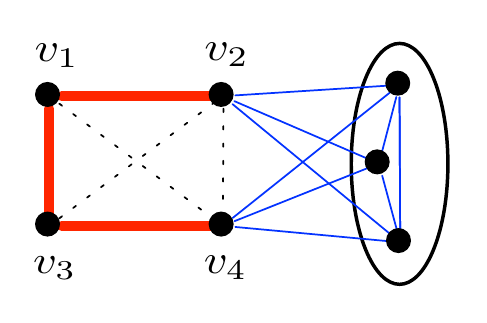}}
    \caption{\quad The structure of $H$ in Cases~$1a$ and~$1b$. The $G_2$ edges are bold.}
    \label{fig:adj1}
\end{figure}
%So we get that there is at least another $G_2$ edge of $H$ joined to $v_2$ or $v_3$ (Figure~\ref{fig:case11}).

% \paragraph

\noindent
{\bf Case~$1b$.}\enskip
Assume that there is a $G_2$ edge, say $v_3v_4$, connecting $\{ v_1,v_2,v_3\}$ to $K\setminus \{ v_1,v_2,v_3\}$ such that  $v_2v_4\notin E(H)$.

According to~(\ref{eq3.2}) the set $A:=\{ v_1, ..., v_4\}$ spans only these three $G_2$ edges, $v_1$ and $v_3$ are degree 2 vertices, and
$(K\setminus A)\cup \{ v_i\}$ are complete graphs for $i\in \{ 2,4\} $.
Since $H$ must contain triangles we have $|K\setminus A|\geq 2$ and $H$ does not contain further $G_2$ edge (see Figure~\ref{fig:adj1}).
Suppose that $y_1\geq y_3$.
We obtain that
 $$
   \frac{\partial}{\partial y_2}\w(G_2, {\mathbf y})=y_1 \geq \frac{\partial}{\partial y_4}\w(G_2, {\mathbf y})=y_3,
  $$
and
$$
\frac{\partial}{\partial y_2}\w(G_1, {\mathbf y})= y_1 +\sum_{\ell> 4}  y_\ell\geq \frac{\partial}{\partial y_4}\w(G_1, {\mathbf y})=
   y_3+\sum_{\ell> 4} y_\ell.
$$
Since $v_2$ and $v_4$ are non-adjacent, this contradicts Condition (\ref{eq3}).

\paragraph{\bf Case~$1c$.}\enskip
Assume that there is a $G_2$ edge, say $v_3v_4$, connecting $\{ v_1,v_2,v_3\}$ to $K\setminus \{ v_1,v_2,v_3\}$ such that  $v_2v_4\in E(H)$.

According to~(\ref{eq3.2}) the set $A:=\{ v_1, ..., v_4\}$ spans only these four edges, $v_1$ and $v_3$ have degree 2, and
$K\setminus \{ v_1, v_3\}$ is a complete graph of size at least $3$ (see Figure~\ref{fig:adj2}).
$H$ does not contain other $G_2$ edges. We have
$$
 \w(G_1,{\mathbf y})= (y_1+y_4)(y_2+y_3) + (y_2+y_4)\left(\sum_{\ell> 4} y_\ell\right) + \underset{i>j>4}{\sum\sum}\,\, y_iy_j
 %%% \sum_{i>j}\sum_{>4} y_iy_j,
  $$
and
$$
 \w(G_2,{\mathbf y})= y_1y_2+ y_1y_3 + y_3y_4.
  $$
Substitute ${\mathbf y}':={\mathbf y}'(t)= {\mathbf y} + t({\mathbf e}_1+{\mathbf e}_2-{\mathbf e}_3-{\mathbf e}_4)$ into the above equations (Figure~\ref{fig:adj2}).
Note that ${\mathbf y}'\in S_n$ if $t\in I:= [\max\{-y_1, -y_2\}, \min\{ y_3, y_4\}]$.
We get  $\w(G_1, {\mathbf y}')=\w(G_1,{\mathbf y})$ and
$$
 \w(G_2, {\mathbf y}')-\w(G_2,{\mathbf y})= t^2+t(y_2-y_4).
  $$
The right hand side is a convex polynomial of $t$ and it takes its maximum on $I$ in one of the endpoints.
Taking this optimal $t$ we obtain that
$ \max _{t\in I} \w(G_2, {\mathbf y}') >\w(G_2,{\mathbf y})$ and $|\supp({\mathbf y}')|< |\supp({\mathbf y})|$, a contradiction.
This completes the proof of Claim~\ref{cl:1} that $H$ has no adjacent $G_2$ edges.
\end{proof}
\begin{figure}[t]
    \centering
    {\includegraphics[width=0.32\textwidth]{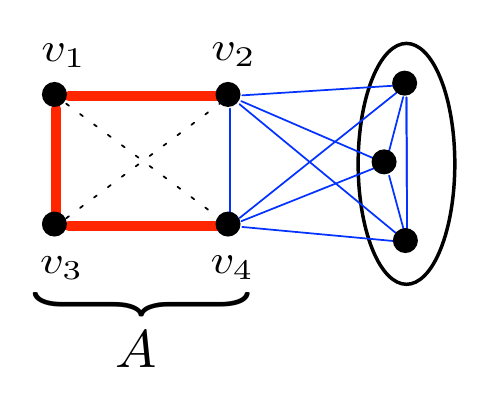}}
    \hspace{16mm}
    {\includegraphics[width=0.32\textwidth]{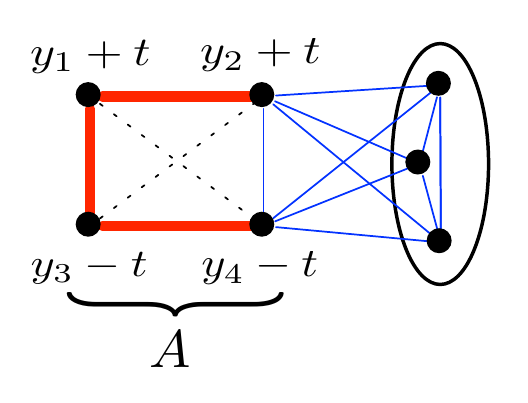}}
    \caption{\quad The structure of $H$ in Case~$1c$, \hspace{10mm} and  the change of the weights.}
    \label{fig:adj2}
\end{figure}

\begin{claim}\label{cl:2}
There are no two parallel edges of $G_2[K]$.
\end{claim}

\begin{proof}[Proof of Claim~\ref{cl:2}.]
According to Claim~\ref{cl:1}, $G_2[K]$ is a matching, $\{v_1v_2$, $ v_3v_4, \dots$, $v_{2k-1}v_{2k}\}$. We will show $k=1$.
Assume, to the contrary, that $v_1v_2$ and $v_3v_4$ are two disjoint $G_2$ edges of $H$.

Define $A:=\{ v_1, \dots, v_4 \}$. Since $v_1v_2$ and $v_3v_4$ are two non-triangular edges, the set $A$ can contain at most two more edges of $H$, and those should be disjoint to each other.
So without losing generality, we may assume that $v_1v_4$ and $v_2v_3 \not\in E(G_1)$, (cf., Figure~\ref{fig:case2}).

Let $B_i:=\{ v\in K\setminus A: vv_i\in E(H)\}$ for $1\leq i\leq 4$.
We claim that $B_1=B_3$. Indeed, if $v_5\in B_1$ then $v_2v_5\not\in E(G_1)$, otherwise $\{ v_1,v_2,v_5\}$ forms a triangle. Then $v_5v_3\in E(H)$ otherwise $\{ v_2, v_3,v_5\}$ forms an independent set. Hence $v_5 \in B_3$, implying $B_1 \subseteq B_3$. By symmetry $B_3 \subseteq B_1$, we obtain $B_1=B_3$ and similarly $B_2=B_4$.

Since $v_1v_2$ is a $G_2$ edge we have $B_1\cap B_2=\emptyset$ (actually, $\{A, B_1, B_2\}$ is a partition of $K$). We distinguish two cases.
%Since $v_1v_2$ is a $G_2$ edge we have $B_1\cap B_2=\emptyset$.
%We claim that $\{A\cup B_1\cup B_2\}$ is a partition of $K$.
%Indeed, if $v_5\not\in (A\cup B_1)$ then $v_5v_4\in E(H)$, otherwise $\{ v_5,v_1,v_4\}$ forms an independent set.
%Then $v_5v_3\notin E(H)$ otherwise $\{ v_5, v_3,v_4\}$ forms a triangle.
%Then the set $\{ v_5,v_3,v_2\}$ contains two non-edges ($v_2v_3$ and $v_5v_3$) so we get $v_5v_2$ is an edge, i.e., $v_5\in B_2$ (Figure~\ref{fig:par1}).
%This implies that $A$, $B_1$ and $B_2$ are indeed forming a partition of $K$.  The same is true for $A$, $B_3$ and $B_4$.
%We also obtained that $v_5$ must belong to $B_4$, so $B_2\subseteq B_4$. Similarly, $B_1\subseteq B_3$.

\begin{figure}[ht]
    \centering
     {\includegraphics[width=0.38\textwidth]{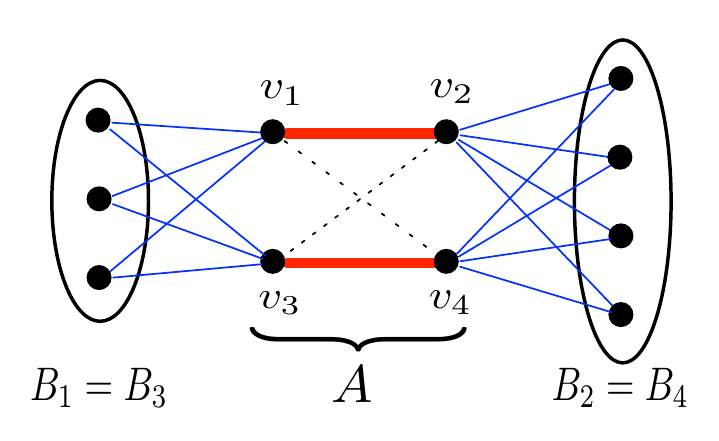}}
    \hspace{10mm}
    {\includegraphics[width=0.38\textwidth]{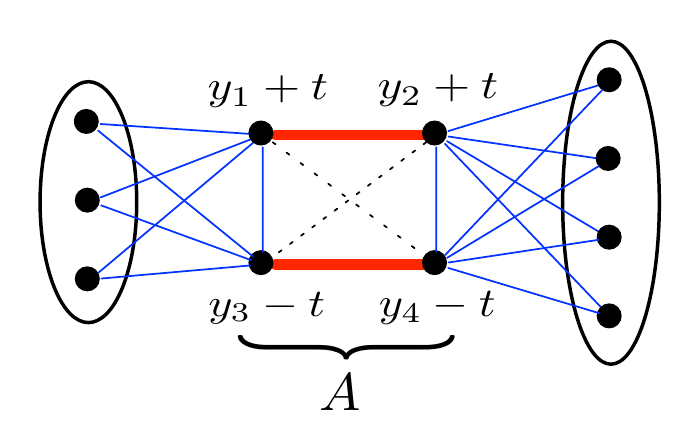}}
    \caption{\quad The structure of $H$ in Case~$2$, \hspace{8mm} and  the change of the weights in Case $2b$.\label{fig:case2}}
\end{figure}

\paragraph{\bf Case~$2a$.}\enskip
Assume first that $v_1v_3 \notin E(H)$.

Suppose that $y_2\geq y_4$.
Since no $G_2$-edge joins $A$ to $K\setminus A$
  and $B_1=B_3$ we obtain that
 $$
   \frac{\partial}{\partial y_1}\w(G_2, {\mathbf y})=y_2 \geq \frac{\partial}{\partial y_3}\w(G_2, {\mathbf y})=y_4,
  $$
and
$$
\frac{\partial}{\partial y_1}\w(G_1, {\mathbf y})= y_2 +\sum_{y_\ell\in B_1}  y_\ell\geq \frac{\partial}{\partial y_3}\w(G_1, {\mathbf y})=
   y_4+\sum_{y_\ell\in B_1} y_\ell.
$$
Since $v_1$ and $v_3$ are non-adjacent, this contradicts (\ref{eq3}).

So we may assume that $A$ contains the edge $v_1v_3$.
By symmetry, we may assume that $A$ contains the edge $v_2v_4$, too.

\paragraph{\bf Case~$2b$.}\enskip
Finally, $A$ contains the edges $v_1v_3$ and $v_2v_4$ (see Figure~\ref{fig:case2}).

We have
%\begin{multline}
%${}$\quad\quad
%$
% \w(G_1,{\mathbf y})= (y_1+y_4)(y_2+y_3)$ \\
% $+ (y_1+y_3)\left(\sum_{y_\ell\in B_1} y_\ell\right) + (y_2+y_4)\left(\sum_{y_\ell\in B_2} y_\ell\right) +
% %% \sum\!\!\!\!\!\!\!\! \sum_{v_i,v_j\notin A\,\,\,\,\quad{}} y_iy_j,
% \underset{v_i,v_j\notin A}{\sum\sum}\,\, y_iy_j,$
%\quad\quad${}$\label{eq3.2}
%\end{multline}

$$
\begin{array}{lll}
 \w(G_1,{\mathbf y})&=& (y_1+y_4)(y_2+y_3) \\[2mm]
& & + (y_1+y_3)\left(\sum_{y_\ell\in B_1} y_\ell\right) + (y_2+y_4)\left(\sum_{y_\ell\in B_2} y_\ell\right) +
 %% \sum\!\!\!\!\!\!\!\! \sum_{v_i,v_j\notin A\,\,\,\,\quad{}} y_iy_j,
 \underset{v_i,v_j\notin A,\enskip v_iv_j \in E(H)}{\sum\sum}\,\, y_iy_j,
 \end{array}
  $$
and
$$
 \w(G_2,{\mathbf y})= y_1y_2 + y_3y_4+\dots+y_{2k-1}y_{2k}.
  $$
Substitute ${\mathbf y}':={\mathbf y}'(t)= {\mathbf y} + t({\mathbf e}_1+{\mathbf e}_2-{\mathbf e}_3-{\mathbf e}_4)$ into the above equations.
Note that ${\mathbf y}'\in S_n$ if $t\in I:= [\max\{-y_1, -y_2\}, \min\{ y_3, y_4\}]$.
We get  $\w(G_1, {\mathbf y}')=\w(G_1,{\mathbf y})$ and
$$
 \w(G_2, {\mathbf y}')-\w(G_2,{\mathbf y})= 2t^2+t(y_1+y_2-y_3-y_4).
  $$
The right hand side is convex, it takes its maximum on $I$ in one of the endpoints.
Taking this optimal $t$ we obtain that
$ \max _{t\in I} \w(G_2, {\mathbf y}') >\w(G_2,{\mathbf y})$ and $|\supp({\mathbf y}')|< |\supp({\mathbf y})|$, a contradiction. This completes the proof of Claim~\ref{cl:2}.
\end{proof}

\medskip
\noindent
{\it The end of the proof of Lemma~\ref{lem:triw}.}\quad
Claims~\ref{cl:1} and~\ref{cl:2} imply that $H$ has a unique $G_2$ edge.  %%% , $e$.
We claim that the vertices in $H$ which are not adjacent to any $G_2$ edge of $H$ induce a clique.
To see this, consider  two such vertices $v_i$ and $v_j$.
We have $ \frac{\partial}{\partial y_i}\w(G_2, {\mathbf y})=0= \frac{\partial}{\partial y_j}\w(G_2, {\mathbf y})$ so the inequalities of (\ref{eq3}) hold. Therefore $v_i$ and $v_j$ must be adjacent to avoid a contradiction.
}\end{proof}

%%%%%%%%%%%%%%%%%%%%%%%%%%%%%%%%%%%%%%%%%%%%%%%%%%%%%%%%%%%%%%%%%%%%%%%%%%%%%%%%%%%%

\section{A continuous lower bound for the number of triangular edges}
\label{sec:lb}

In this section, by using Lemma~\ref{lem:triw}, we will prove the %%%\textendash
  main result of this paper, i.e., Theorem~\ref{thm:th2}.
Recall that  %%% $g(n,e):=  \Tr (G)$ for some $G\in {\mathcal G}(n,e)$.
                             $ g(n,e): =\min\{e-bc : G \in {\mathcal G}(a,b,c)$ with $e(G)\geq e, \, a+b+c=n\}$
  (see~\eqref{eq2}).
We define $t(n,e)$ to be a real valued version of $g(n,e)$ as follows,
\begin{equation}\label{eq41}
t(n,e):=\min\{e-bc: a+b+c=n,\, a,b,c \in {\mathbb{R}}_+,\,\, \frac{1}{2}a^2+ab+bc \ge e\}.
   \end{equation}
Obviously, $t(n,e)\leq g(n,e)$ for $n^2/4\leq e \leq {n \choose 2}$.
%However, the integer and the real valued functions are not too far from each other
Furthermore,
\begin{equation}\label{lem8}
g(n, e)- \coef1 n \leq t(n,e).
\end{equation}
Indeed, suppose that $(a,b,c)\in {\mathbb{R}}_+^3$ yields the optimal value, $t(n,e)=e-bc$.
It is a straightforward calculation to show that the choice of
$(a',b',c'):=(\lceil a +1\rceil, \lceil b \rceil, n-a'-b')$ satisfies (\ref{eq2})
and the difference between $(e-b'c')$ and $(e-bc)$ is at most $\coef1 n$.

We cannot prove Conjecture~\ref{conj1} that $g(n,e)\leq\Tr(n,e)$ (i.e., that they are equal),
 but as an application of Lemma~\ref{lem:triw} we will show that $t(n,e)$ is a lower bound for $\Tr(n,e)$.
\begin{theorem}\label{th7}
For $e>n^2/4$ we have
$t(n, e) \leq \Tr(n,e)$.
\end{theorem}

%%%\begin{proof}%
\noindent{\bf Proof}. % of Theorem~\ref{th7}.]
Suppose that $G_1$ is a graph with $n$ vertices, $e$ edges and minimum number of edges in triangles, i.e., $G_1$ has $\Tr(n,e)$ triangle edges.
Let $G_2$ be the subgraph of $G_1$ consisting of the edges not in any triangle of $G_1$.
Consider the vector $(1/n){\mathbf j}=(1/n, 1/n, \dots, 1/n) \in  {\mathbb{R}}^n$.
By Lemma~\ref{lem:triw} there exists a ${\mathbf y}=(y_1,\dots,y_n)\in S_n$ with support $K$ such that
$G_2[K]$ consists of a single edge, say $v_1v_2$. Moreover
\begin{equation}\label{eq43}
%%%    \frac{e}{n^2} = \w(G_1,(1/n){\mathbf j})\leq \w(G_1,{\mathbf y})\quad {\rm and }\quad
    \frac{e}{n^2} = \w(G_1,(1/n){\mathbf j})\leq \w(G_1,{\mathbf y})
   \end{equation}
 and
 \begin{equation}\label{eq44}
    \frac{e-\Tr(n,e)}{n^2} = \w(G_2,(1/n){\mathbf j})\leq \w(G_2,{\mathbf y})=y_1y_2.
   \end{equation}
Assume that $y_1 \geq y_2$ and define $a:= \left(\sum_{k\neq 1,2} y_k\right)n$, $b:=y_1n$, $c:=y_2n$.
Then \eqref{eq44} yields that  $\Tr(n,e)\geq e - bc$.
We claim that the reals $a,b$, and $c$ satisfy the constraints in~\eqref{eq41}, hence
 $e-bc  \geq t(n,e)$, completing the proof.

Indeed, since $v_1v_2$ is not in any triangle, $N(v_1)\cap N(v_2) = \emptyset$, we get from~\eqref{eq43} that
\begin{eqnarray*} % \label{eq3}
  \frac{e}{n^2} & \leq & \w(G_1,{\mathbf y})\\[2mm]
  {} & = & y_1y_2 + y_1(\sum_{y_k\in N(v_1), k\neq 2}y_k) + y_2(\sum_{y_k\in N(v_2), k\neq 1}y_k) +
     %%% \sum_{i<j}\sum_{i,j\neq 2,3} y_iy_j \\
     \underset{i<j, \, \, i,j\neq 1,2}{\sum\sum}\,\, y_iy_j \\[2mm]
  {} &\leq & \frac{bc}{n^2}+ \frac{b}{n}\times \frac{a}{n} + \frac{1}{2}(\frac{a}{n})^2.    {} \hspace{66mm} \qed
\end{eqnarray*}%
%%% We are done.
% \end{proof}

%%%%%%%%%%%%%%%%%%%%%%%%%%%%%%%%%%%%%%%%%%%%%%%%%%
%%%%%%%%%%%%%%%%%%%%%%%%%%%%%%%%%%%%%%%%%%%%%%%%%%
\section{Further problems, minimizing $C_{2k+1}$ edges}
\label{sec:future}

%Erd\H{o}s, Faudree, and Rousseau~\cite{EFR92} also considered a conjecture of Erd\H{o}s

In addition to the question of minimizing the number of triangular edges, Erd\H{o}s, Faudree and Rousseau~\cite{EFR92} also considered a conjecture of Erd\H{o}s~\cite{Erdos97} regarding pentagonal edges asserting that a graph on $n$ vertices and at least  $\lfloor n^2/4\rfloor+1$  edges has at most  $n^2/36+O(n)$ non-pentagonal edges.
This value can be obtained by considering a graph having two components, a complete graph on $[2n/3]+1$ vertices and a complete bipartite
graph on the rest.
This conjecture was mentioned in the papers of Erd\H{o}s~\cite{Erdos97} and also in the problem book of Fan Chung and Graham~\cite{CG97}.

Erd\H{o}s, Faudree, and Rousseau~\cite{EFR92} proved that if $G$ is a graph with $n$ vertices and at least $\lfloor n^2/4\rfloor +1 $ edges then for any fixed $k\geq 2$
 at least $\frac{11}{144}n^2-O(n)$ edges of $G$ are in cycles of length $2k+1$.
So there is  a jump of $\Omega(n^2)$ in the number of $C_5$-edges, while the construction
  of ${\mathcal G}(n,e)$ shows that for $K_3$-edges the change is smoother, $\Tr (n, n^2/4+ x)= O(n\sqrt{x})$.

In a forthcoming paper~\cite{FM14+} we give an example of graphs with $\lfloor n^2/4\rfloor +1$ edges and
 $n^2/8(2+\sqrt{2}) + O(n) = n^2/27.31...$ non-pentagonal edges, disproving Erd\H{o}s' conjecture.
Using the weighted symmetrization method  we show that this coefficient is asymptotically the best possible for
$e> (n^2/4)+o(n^2)$.
On the other hand, we asymptotically establish the conjecture of Erd\H{o}s that for every $k\geq 3$,
 the maximum number of non-$C_{2k+1}$ edges in a graph of size exceeding $(n^2/4)+o(n^2)$ is at most $n^2/36 +o(n^2)$,
 as in the graph of two-components described above.

More generally, given a graph $F$, one can define $h(n,e,F)$ as the minimum number of $F$-edges
   among all graphs  of  $n$  vertices and  $e$  edges.
In a forthcoming paper~\cite{FM14+} we asymptotically determine  $h(n,\lambda n^2, F)$  for any
 fixed  $\lambda$,  when $1/4 < \lambda < 1/2$ and $F$ is 3-chromatic.
Many problems, e.g., an $F$ with a higher chromatic number, or natural generalizations for hypergraphs remain open.

%%%%%%%%%%%%%%%%%%%%%%%%%%%%%%%%%%%%%%%%%%%%%%%%%%
%%%%%%%%%%%%%%%%%%%%%%%%%%%%%%%%%%%%%%%%%%%%%%%%%%

\medskip
\noindent
{\it  A remark on very dense graphs.} \quad
One can verify Conjecture~\ref{conj1} for $n\leq 8$ and in general for $e\geq {n \choose 2}-(3n-13)$.
This and (\ref{eq1}) yield the exact value of $\Tr(n,e)$ for all pairs with $n\leq 10$ except  $\Tr(10,27)$.
More details can be found in the {\tt arXiv} version~\cite{FM15}.

\medskip
\noindent
{\it A remark on  keeping equalities.} \quad
Taking ${\mathbf a}:= {\mathbf e}_\ell^T \in {\mathbb{R}}^{d+1}$ for some $1\leq \ell \leq d$
instead of ${\mathbf a}:= (1,\dots,1,0)^T \in {\mathbb{R}}^{d+1}$ in the proof of
Lemma~\ref{lem:geometry} one can obtain a sharper version of it.
Namely,  there exists a non-zero vector ${\mathbf z}\in {\mathbb{R}}^{d+1}$ such that  $\sum_{1\leq i\leq d+1} z_i=0$
 and ${\mathbf a}_\ell^T {\mathbf z} \geq 0$, but  ${\mathbf a}_i^T {\mathbf z} = 0$ for every $1 \leq i \leq d$, $i\neq \ell$.

This sharper version of Lemma~\ref{lem:geometry} yields a sharper version of  Theorem~\ref{thm:ind}.
Namely, there exists an appropriate vector ${\mathbf y}\in S_n$ such that
   $\w(G_\ell,{\mathbf x})\leq \w(G_\ell,{\mathbf y})$ and $
   \w(G_i,{\mathbf x})= \w(G_i,{\mathbf y})$ for every $1 \le i \le d$, $i\neq \ell$.

Then the proof of  Lemma~\ref{lem:triw} can be adjusted so that given $\ell\in \{ 1,2\}$ one can find
 an appropriate  vector ${\mathbf y}\in S_n$
such that $\w(G_\ell,{\mathbf x})= \w(G_\ell,{\mathbf y})$ and $\w(G_{3-\ell},{\mathbf x})\leq \w(G_{3-\ell},{\mathbf y})$.

%%%%%%%%%%%%%%%%%%%%%%%%%%%%%%%%%%%%%%%%%%%%%%%%%%
\medskip
\noindent
{\it Acknowledgment.} \quad
The authors are very thankful for the referee for helpful comments.
 %%%%%%%%%%%%%%%%%%%%%%%%%%%%%%%%%%%%%%%%%%%%%%%%%%

\bigskip
\noindent
{\bf New developments} (as of May 2016). \quad
Since the first public presentations of our results (e.g., in the Combinatorics seminar of  the Department Mathematics and Computer Science at Emory University,
December 6, 2013, in the Oberwolfach Combinatorics Workshop, Jan 5--11, 2014) and
posting the present manuscript on {\tt arXiv}~\cite{FM15} on November 4, 2014, there were (at least) two remarkable achievements.

Gruslys and Letzter~\cite{GL16+} using a refined version of the symmetrization method proved that
 there exists an $n_0$ such that $\Tr(n,e)=g(n,e)$ for all $n> n_0$.
The second part of our Conjecture~\ref{conj1}, namely that the extremal graph should be from a ${\mathcal G}(a,b,c)$, is still open.

Grzesik, P.~Hu, and Volec~\cite{GHV16+} using Razborov's flag algebra method
 showed that every $n$-vertex graph with 
  $ \lfloor n^2/4\rfloor +1$ edges has at least $(n^2/4) -  n^2/8(2+\sqrt{2}) -\varepsilon n^2$
  pentagonal edges for $n> n_0(\varepsilon)$ for every $\varepsilon > 0$.
They also proved that those graphs have at most  
 $n^2/36 +\varepsilon n^2$ $C_{2k+1}$-edges 
  for $n> n_k(\varepsilon)$ for every $\varepsilon > 0$ and $k\geq 3$.
In~\cite{FM14+} we were able to prove the same results only for graphs with
$ \lfloor n^2/4\rfloor+ \varepsilon n^2$ edges (for $n> n_0(k, \varepsilon)$, $k\geq 2$).
Let's close with a slightly corrected version of Erd\H os conjecture.

\begin{conj}\label{conj2}
Suppose that $G$ is an $n$-vertex graph with $e$ edges, such that $e> n^2/4$ and it has the minimum number of  $C_{2k+1}$-edges, $k\geq 3$, $n> n_k$.
Then $G$ is connected and has two blocks, one of them is a complete bipartite graph
  and the other one is almost complete.
\end{conj}

%\bibliographystyle{abbrv}
%\bibliography{ref_Erdos}

\begin{thebibliography}{1}

\bibitem{CG97}
F.~Chung and R.~Graham.
\newblock {\em Erd{\H o}s on graphs}.
\newblock His legacy of unsolved problems.
\newblock A. K. Peters, Ltd., Wellesley, MA, 1998.

\bibitem{Erdos97}
P.~Erd{\H{o}}s.
\newblock Some recent problems and results in graph theory.
\newblock {\em Discrete Mathematics} 164, 81--85, 1997.
%%% \newblock The Second Krakow Conference on Graph Theory (Zgorzelisko, 1994).

\bibitem{EFR92}
P.~Erd{\H{o}}s, R.~J. Faudree, and C.~C. Rousseau.
\newblock Extremal problems involving vertices and edges on odd cycles.
\newblock {\em Discrete Mathematics} 101, 23--31, 1992.
%%% \newblock Special volume to mark the centennial of Julius Petersen's ``Die  Theorie der regul{\"a}ren Graphs'', Part II.

\bibitem{FM15}
Z.~F\"{u}redi and Z.~Maleki.
The minimum number of triangular edges and a symmetrization for multiple graphs.
\newblock {\tt  arXiv:1411.0771}, 8 pages.
Posted on November 4, 2014.

\bibitem{FM14+}
Z.~F\"{u}redi and Z.~Maleki.
\newblock A proof and a counterexample for a conjecture of Erd{\H{o}}s
  concerning the minimum number of edges in odd cycles.
\newblock {\em Manuscript}.

\bibitem{GL16+}
V.~Gruslys and S.~Letzter.
\newblock Minimising the number of triangular edges.
\newblock {\tt arXiv:1605.00528}, 43 pages.
Posted on May 2, 2016.

\bibitem{GHV16+}
A.~Grzesik, P.~Hu and J.~Volec.
\newblock Minimum number of edges that occur in odd cycles.
\newblock {\tt arXiv:1605.09055},  24 pages. Posted on May 29, 2016.

\bibitem{MS65}
T.~S. Motzkin and E.~G. Straus.
\newblock Maxima for graphs and a new proof of a theorem of {T}ur\'an.
\newblock {\em Canad. J. Math.} 17,  533--540, 1965.

\bibitem{Turan}
P.~Tur\'an.
\newblock On an extremal problem in graph theory.
\newblock {\em Matematikai \'{e}s Fizikai Lapok (in Hungarian)}  48, 436--452, 1941.

\bibitem{Zykov49}
A.~A. Zykov.
\newblock On some properties of linear complexes.
\newblock {\em Mat. Sbornik N.S. (in Russian)}, 24(66), 163--188, 1949.

\end{thebibliography}

\end{document}